\documentclass[11pt, leqno]{amsart}
\usepackage{amsmath,amssymb,amsthm,enumerate,bbm}

\usepackage{graphicx}
\usepackage{booktabs}
\usepackage[font=small]{caption} 
\usepackage[flushleft]{threeparttable}
\usepackage{mathtools,breqn}
\usepackage{url}
\def\arXiv#1{arXiv:\href{http://arXiv.org/abs/#1}{#1}}

\def\MR#1{MR\href{http://www.ams.org/mathscinet-getitem?mr=#1}{#1}}

\def\spine{1.1in}
\usepackage[
heightrounded,
top=1.15in,
bottom=1.2in,
inner=\spine,
outer=\spine 
]{geometry} %

\usepackage[colorlinks,linkcolor=blue,citecolor=magenta,urlcolor=black,hypertexnames=false]{hyperref}

\def\N{\mathbb{N}}
\def\Z{\mathbb{Z}}

\def\R{\mathbb{R}}

\renewcommand{\v}[1]{\ensuremath{\mathbf{#1}}}

\numberwithin{equation}{section}

\newtheorem{theorem}{Theorem}[section]
\newtheorem{lemma}[theorem]{Lemma}
\newtheorem{remark}[theorem]{Remark}

\newtheorem{proposition}[theorem]{Proposition}

\theoremstyle{definition}

\newtheorem{claim}[theorem]{Claim}

\title{Balian-Low Theorems in Several Variables}
\date{\today}
\author{Michael Northington V}
\author{Josiah Park}
\address{School of Mathematics, Georgia Institute of Technology, Atlanta, GA 30332} 
\email{mcnv3@gatech.edu}
\address{School of Mathematics, Georgia Institute of Technology, Atlanta, GA 30332}
\email{j.park@gatech.edu}


\begin{document}

\begin{abstract}
Recently, Nitzan and Olsen showed that Balian-Low theorems (BLTs) hold for discrete Gabor systems defined on $\Z_d$.  Here we extend these results to a multivariable setting. Additionally, we show a variety of applications of the Quantitative BLT, proving in particular nonsymmetric BLTs in both the discrete and continuous setting for functions with more than one argument. Finally, in direct analogy of the continuous setting, we show the Quantitative Finite BLT implies the Finite BLT.

\end{abstract}

\maketitle

\section{Introduction}\label{sec:Intro}

Gabor systems are fundamental objects in time-frequency analysis. Given a set $\Lambda \subset \R^{2l}$ and a function $g \in L^2(\R^l)$, the Gabor system $G(g, \Lambda)$ is defined as 
\begin{eqnarray*} G(g,\Lambda) &=& \{ g(x-m) e^{2\pi i n\cdot x}\}_{(m,n) \in \Lambda}. \end{eqnarray*}
When $\Lambda$ is taken to be $\Z^{2l}$, $G(g)=G(g,\Z^{2l})$ is referred to as the \emph{integer lattice Gabor system} generated by $g$. The Balian-Low theorem (BLT) and its generalizations are uncertainty principles concerning the generator $g$ of such a system in the case that $G(g,\Lambda)$ forms a Riesz basis. 
\begin{theorem}[BLTs]\label{thm:BLT}
	Let $g\in L^2(\R)$ and suppose that the Gabor system $G(g)=G(g,\Z^2)$ is a Riesz basis for $L^2(\R)$.   
	\begin{enumerate}
		\item[(i)] If $1<p<\infty$ and $\frac{1}{p}+\frac{1}{q}=1$, then either,  
		\begin{eqnarray*} \int_\R |x|^p|g(x)|^2 dx\  =\ \infty\  \text{   or   }\   \int_\R |\xi|^q | \widehat{g}(\xi)|^2 d\xi\  =\  \infty. \end{eqnarray*}
		\item[(ii)] If $g$ is compactly supported, then 
		\begin{eqnarray*} \int_\R |\xi||\widehat{g}(\xi)|^2 d\xi &=&\infty.\end{eqnarray*}
		This part also holds with $g$ and $\widehat{g}$ interchanged.
	\end{enumerate}
\end{theorem}

The first theorem, stated independently by Balian \cite{Bal81} and Low \cite{Low}, was the \emph{symmetric} (i.e., $p=q=2$) case of the theorem above and originally was stated only for orthonormal bases.  The first proofs contained a common error, and a new, correct proof came later from Battle \cite{Bat88}. Soon afterwards, Coifman, Daubechies, and Semmes \cite{Dau} completed the argument in the original proofs and extended the result to all Riesz bases. The second part of Theorem \ref{thm:BLT} was originally given by Benedetto, Czaja, Powell, and Sterbenz \cite{BCPS}, while Gautam \cite{G} extended the BLT to the full range of \emph{nonsymmetric} (i.e., $p \neq q$) cases above.

The Balian-Low Theorem has been generalized in many ways. For example, Gr\"ochenig, Han, Heil, and Kutyniok \cite{GHHK} extended the symmetric Balian-Low theorem to multiple variables.
\begin{theorem}[Theorem 9, \cite{GHHK}]\label{thm:SymmetricBLTHD}
	Let $g \in L^2(\R^l)$ and consider the Gabor system $G(g, \Z^{2l})= \{ g(x-m) e^{2\pi i n\cdot x} \}_{(m,n)\in \Z^{2l}}$.  If $G(g,\Z^{2l})$ is a Riesz basis for $L^2(\R^l)$, then  for any $1 \le k \le l$, either 
	\begin{eqnarray*} \int_{\R^l} |x_k|^2|g(x)|^2 dx\ =\ \infty\  \text{ or }\   \int_{\R^l} |\xi_k|^2 | \widehat{g}(\xi)|^2 d\xi\  =\  \infty. \end{eqnarray*}
\end{theorem}

Another important generalization is the following \emph{Quantitative BLT} of Nitzan and Olsen, which quantitatively bounds the time-frequency localization of a square-integrable function. 
\begin{theorem}[Theorem 1, \cite{NOQuant}] \label{thm:QuantBLT}
	Let $g \in L^2(\R)$ be such that $G(g)$ is a Riesz basis for $L^2(\R)$.  Then, for any $R,L\ge 1$, we have 
	\begin{eqnarray}\label{eqn:quantBLT}
	\int_{|x|\ge R} |g(x)|^2 dx + \int_{|\xi|\ge L} |\widehat{g}(\xi)|^2 d\xi &\ge& \frac{C}{RL},
	\end{eqnarray}
	where the constant $C$ only depends on the Riesz basis bounds for $G(g)$.  
\end{theorem}
This result has also been extended to Gabor systems in $L^2(\R^l)$ in \cite{Temur}. (See Theorem \ref{thm:ContQuantBLTHD} below.)
The Quantitative BLT is a strong result. In particular, a function satisfying \eqref{eqn:quantBLT} automatically satisfies the conclusions of both parts of Theorem \ref{thm:BLT}.  Later, we will use the Quantitative BLT and its higher dimensional analog to show that nonsymmetric versions of Theorem \ref{thm:SymmetricBLTHD} hold for $\R^l$, $l\geq 2$.

In applications Gabor systems are used in signal analysis to give alternate representations of data with desirable properties. Often it is useful to have window functions which measure locally in time for efficiency while capturing local frequency information simultaneously. Uncertainty theorems like the BLT limit how well localized a window can be in the time and frequency domains. This led Lammers and Stampe \cite{Lammers} to conjecture that finite versions of the BLT should hold for discrete Gabor systems.  The essence of this question was answered in one dimension by Nitzan and Olsen \cite{Nitzan} who showed that versions of both the BLT and the quantitative BLT exist for discrete Gabor systems.

In the finite setting, instead of functions in $L^2(\R)$, complex--valued sequences defined on $\Z_d=\Z/d\Z$ act as the object of study, where $d=N^2$ for some $N\in \N$.  It is sometimes useful to fix representatives of $\Z_d$ in connection with the view of $\Z_{d}$ as a discretization of $\R$, and a convenient choice in what follows is $I_d=[-d/2,d/2)\bigcap \Z=\{ -\lfloor\frac{d}{2}\rfloor, ...,d- \lfloor\frac{d}{2}\rfloor-1\}$. Such sequences $b$ may be thought of as samples of a continuous function $g$ defined on $[-\frac{N}{2}, \frac{N}{2}]$ at integer multiples of $1/N$ so that $b(j)=g(j/N)$ for $j \in I_d$.  

Let $\ell_2^d$ denote the set of complex-valued sequences on $\Z_d= \Z/d\Z$ with the norm
\begin{eqnarray} \|b\|_{\ell_2^d}^2 &=& \frac{1}{N} \sum_{j\in \Z_d} |b(j)|^2\label{l2dnorm}.\end{eqnarray}
With this normalization and the sampling view noted above, $\|b\|_{\ell_2^d}^2$ approximates $\|g\|_{L^2[-\frac{N}{2},\frac{N}{2}]}^2$. We define the \emph{discrete Gabor system generated by }$b$, denoted $G_d(b)$, to be,
\begin{eqnarray*} G_d(b) &=& \{ b(j-nN) e^{2\pi i \frac{mj}{N}}\}_{(n,m)\in \{0,...,N-1\}^2}= \{ b(j-n)e^{2\pi i \frac{mj}{d}}\}_{(n,m)\in (N\Z_d)^2}. \end{eqnarray*}
Here $N\Z_d= \{Nj: j\in\Z\}\mod d$ so that $\#(N\Z_d)=N$. This definition lines up with the definition of $G(g)$ above, as shifting $g$ by $n$ corresponds to shifting $b$ by $nN$, and modulation of $g$, $g(x)e^{2\pi i m x}$, corresponds to a new sequence $b(j)e^{2\pi i m j/N}$.

To formulate the (symmetric) BLT in a finite setting, it is useful to consider an equivalent condition to the conclusion of the BLT which is in terms of the distributional derivatives of $g$ and $\widehat{g}$, $Dg$ and $D\widehat{g}$.  In particular, the condition 
\begin{eqnarray} \int_\R |x|^2|g(x)|^2 dx\ =\ \infty\  \text{ or } \ \int_\R |\xi|^2 |\widehat{g}(\xi)|^2 d\xi\ =\ \infty\label{eqn:BLTConclusion}\end{eqnarray}
is equivalent to
\begin{eqnarray}\label{eqn:derivatives}
Dg \notin L^2(\R) \text{ or } D\widehat{g} \notin L^2(\R). 
\end{eqnarray}
For finite generators, $b\in \ell_2^d$, we instead work with differences, 
\begin{eqnarray*} \Delta b &=&\{ b(j+1)-b(j)\}_{j \in \Z_d}, \end{eqnarray*}
and note that $N \Delta b$ approximates the derivative of $g$. We normalize the discrete Fourier transform of $b$ by
\begin{eqnarray*} \mathcal{F}_d(b)(k) &=& \frac{1}{N} \sum_{j \in \Z_d} b(j) e^{-2\pi i \frac{jk}{d}}, \end{eqnarray*}
so that $\mathcal{F}_d$ is an isometry on $\ell_2^d$.  Then the quantity 
\[ \|N\Delta b\|_{\ell_2^d}^2 + \|N \Delta \mathcal{F}_d(b)\|_{\ell_2^d}^2\]
acts as a discrete counterpart to the expressions in equation \eqref{eqn:derivatives}. Recall that a sequence $\{h_n\}$ is a Riesz basis for a separable Hilbert space, $\mathcal{H}$, if and only if it is complete in $\mathcal{H}$ and there exists constants $0<A\le B<\infty$ such that 
\begin{equation}\label{eqn:RieszBasisDef}
A\left( \sum_{n } |c_n|^2\right) \le \left\| \sum_{n} c_n h_n \right\|_{\mathcal{H}} \le B\left( \sum_{n} |c_n|^2 \right),
\end{equation}
for any sequence (equivalently, a Riesz basis is the image of an orthonormal basis under a bounded invertible operator on $\mathcal{H}$).  Here $A$ and $B$ are referred to as the lower and upper Riesz basis bounds, respectively. We say that $b$ \emph{generates an $A,B$-Gabor Riesz basis} if $G_d(b)$ is a basis for $\ell_2^d$ with Riesz basis bounds $A$ and $B$.

The following \emph{Finite BLT} of Nitzan and Olsen shows optimal bounds on the growth of this quantity for the class of sequences which generate Gabor Riesz bases with fixed Riesz basis bounds.  

\begin{theorem}[Theorem 4.2, \cite{Nitzan}] \label{thm:FiniteBLT}
	For $0<A\le B<\infty$, there exists a constant $c_{AB}>0$, depending only on $A$ and $B$, such that for any $N\ge 2$ and for any $b\in \ell_2^d$ which generates an $A,B$-Gabor Riesz basis for $\ell_2^d$, 
	\begin{eqnarray*} c_{AB}\log(N) &\le& \|N\Delta b\|_{\ell_2^d}^2 + \|N \Delta \mathcal{F}_d(b)\|_{\ell_2^d}^2. \end{eqnarray*}
	Conversely, there exists a constant $C_{AB}$ such that for any $N\ge2$, there exists $b\in \ell_2^d$ which generates and $A,B$-Gabor Riesz basis for $\ell_2^d$ such that 
	\begin{eqnarray*} \|N\Delta b\|_{\ell_2^d}^2 + \|N \Delta \mathcal{F}_d(b)\|_{\ell_2^d}^2 &\le& C_{AB}\log(N). \end{eqnarray*}
\end{theorem}


Nitzan and Olsen also show that the continuous BLT, Theorem \ref{thm:BLT}, follows from this discrete version and that the following \emph{Finite Quantitative BLT} also holds.
\begin{theorem}[Theorem 5.3, \cite{Nitzan}]\label{thm:FiniteQuantBLT}
	Let $A,B>0$.  There exists a constant $C_{AB}>0$ such that the following holds.  Let $N \ge 200\sqrt{B/A}$ and let $b \in \ell_2^d$ generate an $A,B$-Gabor Riesz basis.  Then, for all positive integers $1\le Q, R\le (N/16)\sqrt{A/B}$, we have 
	\begin{eqnarray*} \frac{1}{N} \sum_{j=NQ}^{d-1} |b(j)|^2 + \frac{1}{N} \sum_{k=NR}^{d-1} |\mathcal{F}_d b(k)|^2 &\ge& \frac{C_{AB}}{QR}. \end{eqnarray*}
\end{theorem}

\subsection{Extension to Several Variables}
The first goal of this paper is to extend Theorem \ref{thm:FiniteBLT} and \ref{thm:FiniteQuantBLT} to several variables, which we state below in Theorems \ref{thm:FiniteBLTHD} and \ref{thm:FiniteQuantBLTHD}.  

We consider complex-valued sequences on $\Z_d^l=\Z_d \times \cdots \times \Z_d$ for $l\ge 1$, and we denote the set of all such sequences as $\ell_2^{d,l}$. The view of these sequences as samples of a continuous $g \in L^2([-\frac{N}{2},\frac{N}{2}]^l)$, where $b(\v{j})=g(\v{j}/N)$ for $\v{j}=(j_1,...,j_l)\in I_d^l$ leads to the normalization 
\begin{eqnarray*} \|b\|_{\ell_2^{d,l}}^2 \ =\  \frac{1}{N^l} \sum_{\v{j} \in \Z_d^l} |b(\v{j})|^2 \ =\ \frac{1}{N^l} \sum_{\v{j} \in I_d^l} |b(\v{j})|^2. \end{eqnarray*}
The discrete Fourier transform, $\mathcal{F}_{d,l}$, on $\ell_2^{d,l}$, is given by 
\begin{eqnarray*} \mathcal{F}_{d,l}(b)(\v{k}) &=& \frac{1}{N^l} \sum_{\v{j}\in \Z_d^l} b(\v{j}) e^{-2\pi i \frac{\v{j}\cdot \v{k}}{d}}. \end{eqnarray*}
Under this normalization, $\mathcal{F}_{d,l}$ is an isometry on $\ell_2^{d,l}$.  The Gabor system generated by $b$, $G_{d,l}(b)$ is given by 

\begin{eqnarray*} G_{d,l}(b) \ =\  \{ b(\v{j}-N\v{n}) e^{2\pi i \frac{\v{j}\cdot \v{m}}{N}}\}_{(\v{n},\v{m})\in \{0,...,N-1\}^{2l}}\ =\  \{ b(\v{j}-\v{n})e^{2\pi i \frac{\v{j}\cdot \v{m}}{d}}\}_{(\v{n},\v{m})\in (N\Z_d)^{2l}}. \end{eqnarray*} 

For any $k \in \{1,...,l\}$, let $\Delta_k: \ell_2^{d,l}\rightarrow \ell_2^{d,l}$ be defined by 
\begin{eqnarray*} \Delta_k b (\v{j}) &=& b(\v{j}+\v{e}_k)-b(\v{j}), \end{eqnarray*} 
where $\{\v{e}_k\}_{k \in \{1,...,l\}}$ is the standard orthonormal basis for $\R^l$.  Then $N \Delta_k b$ approximates the partial derivative $\frac{\partial g}{\partial x_k}$.

We have the following generalization of Theorem \ref{thm:FiniteBLT}. 
\begin{theorem} \label{thm:FiniteBLTHD}
	Fix constants $0<A\le B<\infty$.  With the same constants $c_{AB}$ and $C_{AB}$ from Theorem \ref{thm:FiniteBLT}, for $N\ge 2$, $1\le k \le l$, and for any $b \in \ell_2^{d,l}$ which generates an $A,B$-Gabor Riesz basis for $\ell_2^{d,l}$, we have
	\begin{eqnarray*} c_{AB} \log(N) &\le& \|N\Delta_k b\|_{\ell_2^{d,l}}^2 + \|N \Delta_k \mathcal{F}_{d,l}(b)\|_{\ell_2^{d,l}}^2. \end{eqnarray*}
	Conversely, for $N\ge 2$ and $1 \le k \le l$, there exists $b\in \ell_2^{d,l}$ which generates an $A,B$-Gabor Riesz basis such that
	\begin{eqnarray*} \|N\Delta_k b\|_{\ell_2^{d,l}}^2 + \|N \Delta_k \mathcal{F}_{d,l}(b)\|_{\ell_2^{d,l}}^2 &\le& C_{AB} \log(N). \end{eqnarray*}
\end{theorem}

We provide a direct proof of Theorem \ref{thm:FiniteBLTHD} in Section \ref{sec:proofFiniteBLTHD}. In Section \ref{sec:proofFiniteQuantBLTHD}, we extend Theorem \ref{thm:FiniteQuantBLT} in the following way. For simplicity of notation, for $t>0$, we let $\{ |j_k|\ge t\}$ denote the set $\{ \v{j} \in I_d^l: |j_k|\ge t\}$.

\begin{theorem}\label{thm:FiniteQuantBLTHD}
	Let $A, B>0$ and $l \in \N$.  There exists a constant $C>0$ depending only on $A$, $B$, and $l$, such that the following holds.  Let $N\ge 200\sqrt{B/A}$ and let $b \in \ell_2^{d,l}$ generate an $A,B$-Gabor Riesz basis for $\ell_2^{d,l}$.  Then, for any $1 \le k \le l$ and all integers $1\le Q, R\le (N/16) \sqrt{A/B}$, we have 
	\begin{eqnarray*} \frac{1}{N^l} \sum_{|j_k|\ge \frac{NR}{2}} |b(\v{j})|^2 + \frac{1}{N^l} \sum_{|j_k|\ge \frac{NQ}{2}} |\mathcal{F}_{d,l}b (\v{j})|^2  &\ge& \frac{C}{QR}. \end{eqnarray*}
\end{theorem}

\subsection{Finite Nonsymmetric BLTs}
In Section \ref{FQBLTHDapps}, we prove nonsymmetric versions of the finite BLT. In the process, we show that symmetric and nonsymmetric versions of the finite BLT follow as corollaries of the finite quantitative BLT (Theorem \ref{thm:FiniteQuantBLTHD}), as long as $N$ is sufficiently large.

\begin{theorem}[Nonsymmetric Finite BLT]\label{thm:NonsymFiniteBLT}
	Let $A,B>0$ and $1<p,q<\infty$ be such that $\frac{1}{p}+\frac{1}{q}=1$.  There exists a constant $C>0$, depending only on $A, B, p$ and $q$ such that the following holds.  Let $N \ge 200\sqrt{B/A}$.  Then, for any $b\in \ell_2^{d,l}$ which generates an $A,B$-Gabor Riesz basis for $\ell_2^{d,l}$, 
	\begin{eqnarray*} C\log(N) &\le& \frac{1}{N^{l}} \sum_{\v{j} \in I_{d}^l} \left|\frac{j_k}{N}\right|^p |b(\v{j})|^2+ \frac{1}{N^{l}} \sum_{\v{j} \in I_{d}^l} \left|\frac{j_k}{N}\right|^q |b(\v{j})|^2. \end{eqnarray*} 
\end{theorem}

\begin{remark}
	Theorem \ref{thm:NonsymFiniteBLT} gives a finite dimensional version of the nonsymmetric BLT for parameters satisfying $1<p,q<\infty$.  Thus, it is a finite dimensional analog of part (i) of Theorem \ref{thm:BLT} in all dimensions.  In Section \ref{FQBLTHDapps} we extend this result to the case where either $p$ or $q$ is $\infty$, thus giving a finite dimensional analog of part (ii) of Theorem \ref{thm:BLT} for all dimensions.  In the same section, a generalization of this result is demonstrated for pairs $(p,q)$ such that $\frac{1}{p}+\frac{1}{q}\neq 1$.  (See Theorem \ref{thm:NonsymFiniteBLTwithInfinity}.)
\end{remark}

\begin{remark}
	It is readily checked that the $p=q=2$ version of Theorem \ref{thm:NonsymFiniteBLT} is equivalent to Theorem \ref{thm:FiniteBLTHD}, so the proof of Theorem \ref{thm:NonsymFiniteBLT} gives an alternative proof of Theorem \ref{thm:FiniteBLTHD} for $N \ge 200\sqrt{B/A}$.  In particular, the proof shows that the Finite Quantitative BLT implies the finite symmetric (and nonsymmetric) BLT.  
\end{remark}

\subsection{Applications of the Continuous Quantitative BLT}
In Section \ref{FQBLTHDapps}, we also prove several results related to functions of continuous arguments.  We first state the simplest of these results, a generalization of Theorem \ref{thm:SymmetricBLTHD} to nonsymmetric weights.

\begin{theorem}\label{thm:BLTHD}
	Let $g\in L^2(\R^l)$ and suppose that $G(g)=G(g,\Z^{2l})$ is a Riesz basis for $L^2(\R^l)$.   For any $1\le k \le \infty$, the following must hold.
	\begin{enumerate}
		\item[(i)] If $1<p<\infty$ and $\frac{1}{p}+\frac{1}{q}=1$, then either  
		\begin{eqnarray*} \int_{\R^l} |x_k|^p|g(x)|^2 dx\ =\ \infty\  \text{ or }\   \int_{\R^l} |\xi_k|^q | \widehat{g}(\xi)|^2 d\xi \ =\  \infty. \end{eqnarray*}
		\item[(ii)] If $g$ is compactly supported, then 
		\begin{eqnarray*} \int_{\R^l} |\xi_k||\widehat{g}(\xi)|^2 d\xi \ =\ \infty. \end{eqnarray*}
		This part also holds with $g$ and $\widehat{g}$ interchanged.
	\end{enumerate}
\end{theorem}

In addition we are able to show more concrete estimates on the growth of related quantities, and we also may remove the assumption that $\frac{1}{p}+\frac{1}{q}=1$.  

\begin{theorem}\label{thm:QuantBLTCorollaryNonsymmetric}
	Suppose $1\le p,q <\infty$ and let $g\in L^2(\R^l)$ be such that $G(g)=G(g, \Z^{2l})=\{ e^{2\pi i n\cdot x} g(x-m)\}_{(m,n)\in \Z^{2l}}$ is a Riesz basis for $L^2(\R^l)$. Let $\tau=\frac{1}{p}+\frac{1}{q}$. Then, there is a constant $C$ depending only on the Riesz basis bounds of $G(g)$ such that for any $1\le k \le l$ and any $2\le T <\infty$, the following inequalities hold.
	\begin{enumerate}
		\item[(i)] If $\tau=\frac{1}{p}+\frac{1}{q} <1$, then 
		\begin{eqnarray*}  \frac{C(1-2^{\tau-1})}{(1-\tau)} T^{1-\tau} &\le& \int_{\R^l} \min(|x_k|^p, T) |g(x)|^2 dx + \int_{\R^l} \min(|\xi_k|^q,T) |\widehat{g}(\xi)|^2d\xi. \end{eqnarray*}
		\item[(ii)] If $\tau=\frac{1}{p}+\frac{1}{q} =1$, then
		\begin{eqnarray*} C \log(T) &\le& \int_{\R^l} \min(|x_k|^p, T) |g(x)|^2 dx + \int_{\R^l} \min(|\xi_k|^q,T) |\widehat{g}(\xi)|^2d\xi.\end{eqnarray*}
		\item[(iii)] If $\tau=\frac{1}{p}+\frac{1}{q} >1$, then 
		\begin{eqnarray*} \frac{C}{\tau-1} &\le& \int_{\R^l} |x_k|^p |g(x)|^2 dx + \int_{\R^l} |\xi_k|^q |\widehat{g}(\xi)|^2d\xi.\end{eqnarray*}
	\end{enumerate}
\end{theorem}

When the bound $2\le T <\infty$ is replaced by $1<\gamma \le T<\infty$, the bound $\frac{C(1-2^{\tau-1})}{(1-\tau)} T^{1-\tau}$ in part \textit{(i)} can be replaced by $\frac{C(1-\gamma^{\tau-1})}{(1-\tau)} T^{1-\tau}$.  In Section \ref{FQBLTHDapps} we extend this theorem to the case where either $p=\infty$ or $q=\infty$.

The first and second inequalities in Theorem \ref{thm:QuantBLTCorollaryNonsymmetric} quantify the growth of `localization' quantities in terms of cutoff weights of the form $\min(|x_k|^p, T)$.   The $\log$ term in the second inequality shows a connection between the continuous BLT and its finite dimensional versions.  The last inequality, on the other hand, shows that generators of Gabor Riesz bases must satisfy a Heisenberg type uncertainty principle for every $0< p \le 2$.   A similar inequality is known to hold for arbitrary $L^2(\R)$ functions by a result of Cowling and Price \cite{CP}.  However, for generators of Gabor Riesz bases, we have explicit estimates on the dependence of the constant on $\tau$ and the result here is stated for higher dimensions.


\section{Preliminaries: The Zak Transform and Quasiperiodic Functions} \label{xcom2}

The Zak transform is an essential tool for studying lattice Gabor systems.  The discrete Zak transform $Z_{d,l}$ of $b\in \ell_2^{d,l}$ for $(\v{m},\v{n}) \in \Z_d^{2l}$ is given by

\begin{align*} Z_{d,l}(b)(\v{m}, \v{n}) &=  \sum_{\v{j} \in \{0,...,N-1\}^l} b(\v{m}-N\v{j}) e^{2\pi i \frac{\v{n}\cdot \v{j}}{N}}=\sum_{\v{j} \in N\Z_d^l} b(\v{m}-\v{j})e^{2\pi i \frac{\v{n} \cdot \v{j}}{d}}. \end{align*}

The following properties show that $Z_{d,l}(b)$ encodes basis properties of $G_{d,l}(b)$, while retaining information about `smoothness' (see the remark following Proposition \ref{prop:ZakProperties}) of $b$ and $\mathcal{F}_{d,l}(b)$.  Note that $Z_{d,l}(b)(\v{m},\v{n})$ is defined for $(\v{m},\v{n})\in \Z_d^{2l}$ and is $d$-periodic in each of its $2l$ variables.  However, the Zak transform satisfies even stronger periodicity conditions.  
In fact, $Z_{d,l}(b)$ is \emph{$N$-quasiperiodic on $\Z_d^{2l}$}, that is
\begin{eqnarray}
Z_{d,l}(b) (\v{m}+N\v{e}_k, \v{n})&=& e^{2\pi i \frac{n_k}{N}} Z_{d,l}(b)(\v{m},\v{n}), \label{eqn:quasiperiodic}\\
Z_{d,l}(b) (\v{m}, \v{n}+N\v{e}_k)&=& Z_{d,l}(b)(\v{m},\v{n}). \nonumber
\end{eqnarray}
Let $S_N=\{0,...,N-1\}$. Then, the quasi-periodicity conditions above show that  $Z_{d,l}(b)$ is completely determined by its values on $S_N^{2l}$.  

We will use the notation $\ell_2(S_N^{2l})$ to denote the set of sequences $W(\v{m},\v{n})$ defined on $S_N^{2l}$ with norm given by 
\[ \|W\|_{\ell_2(S_N^{2l})}^2 \ =\  \frac{1}{N^{2l}} \sum_{(\v{m}, \v{n}) \in S_N^{2l}} |W(\v{m},\v{n})|^2,\]
where here we keep the variables $\v{m}$ and $\v{n}$ separate due to the connection with the Zak transform.  The normalization is chosen so that if $W$ is a sampling of a function $h(\v{x}, \v{y})$ on $[0,1]^{2l}$, then $\|W\|_{\ell_2(S_N^{2l})}$ approximates the $L^2([0,1]^{2l})$ norm of $h$.  

The Zak transform has many other important properties, some of which we collect in the next proposition. Arguments for these facts are standard and presented in \cite{AGT} and \cite{Nitzan}, for instance.
\begin{proposition}\label{prop:ZakProperties}
	Let $b \in \ell_2^{d,l}$. 
	\begin{enumerate}
		\item[(i)] $Z_{d,l}$ is a unitary mapping from $\ell_2^{d,l}$ onto $\ell_2(S_N^{2l})$. 
		\item[(ii)] A sequence $b\in \ell_2^{d,l}$ generates an $A,B$-Gabor Riesz basis for $\ell_2^{d,l}$ if and only if $Z_{d,l}(b)$ satisfies 
		\begin{eqnarray*} A \ \le\  |Z_{d,l}(b)(\v{m}, \v{n}) |^2 \ \le\  B, \text{ for } (\v{m}, \v{n}) \in \Z_d^{2l}. \end{eqnarray*}
		\item[(iii)] Let $\widehat{b}\ =\ \mathcal{F}_{d,l}(b)$.  Then,
		\begin{eqnarray*} Z_{d,l}(\widehat{b})(\v{m},\v{n})\ =\  e^{2\pi i \frac{\v{m}\cdot \v{n}}{d}} Z_{d,l}(b)(-\v{n},\v{m}). \end{eqnarray*}
		\item[(iv)] For $a,b \in \ell_2^{d,l}$ define $(a \ast b)(\v{k})= \frac{1}{N^l} \sum_{j \in \Z_d^{l}} a(\v{k}-\v{j})b(\v{j})$.  Then, 
		\begin{eqnarray*} Z_{d,l}(a\ast b)(\v{m}, \v{n})\ =\  \frac{1}{N^l} \sum_{\v{j} \in \Z_d^l} b(\v{j}) Z_{d,l}(a)(\v{m}-\v{j},\v{n}) \ =\  (Z_{d,l}(a)\ast_1 b) (\v{m},\v{n}), \end{eqnarray*}
		where $\ast_1$ denotes convolution of $b$ with respect to the first set of variables of $Z_{d,l}(a)$, $\v{m}$, keeping the second set, $\v{n}$, fixed.
		
	\end{enumerate}
\end{proposition}

\begin{remark} We will be interested in the `smoothness' of $b$ and $Z_{d,l}(b)$ for $b \in \ell_2^{d,l}$.  Since these are functions on discrete sets, smoothness is not well defined, but we use the term in relation to the size of norms of certain difference operators defined on $\ell_2^{d,l}$ and $\ell_2(S_N^{2l})$, which mimic norms of partial derivatives of differentiable functions. \end{remark}

For $1\le k \le l$ and any $N$-quasiperiodic function on $\Z_d^l$, let $\Delta_k,\Gamma_k$ be defined as follows:
\begin{eqnarray*}
	\Delta_k W(\v{m},\v{n}) &=& W(\v{m}+\v{e_k},\v{n})-W(\v{m},\v{n}),\\
	\Gamma_k W(\v{m},\v{n}) &=& W(\v{m},\v{n}+\v{e_k})-W(\v{m},\v{n}).
\end{eqnarray*}
For $b \in \ell_2^{d,l}$ define $\alpha_k(b)$ and $\beta_k(b)$ by

\begin{eqnarray*}
	\alpha_k(b)&=& \|N\Delta_k b\|_{\ell_2^{d,l}}^2 + \|N \Delta_k \mathcal{F}_{d,l}(b)\|_{\ell_2^{d,l}}^2,\\
	\beta_k(b)&=& \frac{1}{N^{2l}} \sum_{(\v{m},\v{n})\in S_N^{2l}} |N\Delta_k Z_{d,l}(b)(\v{m},\v{n})|^2+\frac{1}{N^{2l}} \sum_{(\v{m},\v{n})\in S_N^{2l}} |N\Gamma_k Z_{d,l}(b)(\v{m},\v{n})|^2.
\end{eqnarray*}

The following proposition shows that $\alpha_k(b)$ and $\beta_k(b)$ are essentially equivalently sized.  Proposition 4.1 in \cite{Nitzan} proves this for the case $l=k=1$, and it is readily checked that the proof carries over directly to the $l>1$ setting.  
\begin{proposition}\label{prop:alphabeta}
	Let $B>0$ and let $b\in \ell_2^{d,l}$ be such that $|Z_{d,l}(b)(\v{m},\v{n})|^2 \le B$ for all $(\v{m},\v{n})\in \Z_d^{2l}$.  Then, for all integers $N\ge 2$ and any $1 \le k \le l$, we have 
	\begin{eqnarray*} \frac{1}{2} \beta_{k}(b) -8\pi^2 B\ \le\  \alpha_{k}(b) \ \le\  2 \beta_{k}(b) + 8 \pi^2 B. \end{eqnarray*}
\end{proposition}
We thus see that in order to bound $\alpha_k(b)$ as in Theorem \ref{thm:FiniteBLTHD}, it is sufficient to bound $\beta_k(b)$.  For $b \in \ell_2^d= \ell_2^{d,1}$, let $\beta(b)= \beta_{1}(b)$, and let 
\begin{eqnarray*}
	\beta_{A,B}(N)=\inf\{ \beta(b)\},
\end{eqnarray*}
where the infimum is taken over all $b \in \ell_2^d$ such that $b$ generates an $A,B$-Gabor Riesz basis. 
\begin{theorem}[Theorem 4.2, \cite{Nitzan}]\label{thm:betabound1d}
	There exist constants $0<c_{AB}\le C_{AB}<\infty$ such that for all $N\ge 2$, we have 
	\[ c_{AB}\log(N) \ \le\  \beta_{A,B}(N)\ \le\  C_{AB}\log(N).\]
\end{theorem}

To prove the lower bound in this theorem (as is done in \cite{Nitzan}), one may examine the argument of the Zak transform of a sequence $b \in \ell_2^d$ which generates a basis with Riesz basis bounds $A$ and $B$ over finite dimensional lattice-type structures in the square $\{0,...,N\}^2$.  Due to the $N$-quasiperiodicity conditions satisfied by $Z_{d,l}(b)$ this argument is forced to `jump' at some step between neighboring points along these lattice-type sets (See Lemma 3.1 and 3.4 in \cite{Nitzan}). Due to the Riesz basis assumption and part (iv) of Proposition \ref{prop:ZakProperties}, jumps in the argument of $Z_{d,l}(b)$ correspond directly to jumps in $Z_{d,l}(b)$  (see Corollary 3.6 in \cite{Nitzan}). By counting the number of lattice-type sets which are disjoint, a logarithmic lower bound is given for the number of jumps in $Z_{d,l}(b)$ corresponding to jumps in the argument, which gives the lower bound in Theorem \ref{thm:betabound1d}.  

The proof of the upper bound involves an explicit construction of the argument of a unimodular function, $W$, on $S_N^2$.  Since the Zak transform is a unitary, invertible mapping between $\ell_2^{d}$ and $\ell_2(S_N^2)$, there is a corresponding $\tilde{b}\in \ell_2^{d}$ so that $G(\tilde{b})$ is an orthonormal basis (which can be scaled to form a Riesz basis with bounds $A$ and $B$ for any $A$ and $B$) and such that $\tilde{b}$ satisfies $Z_{d}(\tilde{b})=W$.  For this construction, $\beta(\tilde{b})$ can be bounded directly to show the upper bound in the theorem.

\section{Proof of Theorem \ref{thm:FiniteBLTHD}}\label{sec:proofFiniteBLTHD}

Based on Proposition \ref{prop:alphabeta}, to prove Theorem \ref{thm:FiniteBLTHD} it is sufficient to show that Theorem \ref{thm:betabound1d} extends from $\ell_2^d$ to $\ell_2^{d,l}$.  
We show this below, and in particular that by restricting the Zak transform of a multi-variable sequence to the $k^{\text{th}}$ variable in each component, we can directly use Theorem \ref{thm:betabound1d} to prove the multi-variable version of the lower bound. Similarly, we show that by taking suitable products of the constructed $\tilde{b}$ function mentioned above, we can also extend the logarithmic upper bound to higher dimensions.

Let 
\begin{eqnarray*}
	\beta_{A,B,k}(N,l)= \inf\{ \beta_k(b)\},
\end{eqnarray*}
where the infimum is over all $b \in \ell_2^{d,l}$ which generate an $A,B$-Gabor Riesz basis for $\ell_2^{d,l}$.  
\begin{theorem}\label{thm:betaboundhd}
	For the same constants $0<c_{AB}\le C_{AB}<\infty$ as Theorem \ref{thm:betabound1d}, for all $N\ge 2$, and for any $1\le k \le l$, we have
	\[ c_{AB}\log(N) \ \le\  \beta_{A,B,k}(N,l)\ \le\   C_{AB}\log(N).\]
\end{theorem}
\begin{proof} For notational convenience, we show both the lower and upper bound with $k=1$, but a similar argument applies for any $1\le k\le l$.
	
	\textbf{Lower bound:}
	Let $b \in \ell_2^{d,l}$ generate an $A,B$-Gabor Riesz basis for $\ell_2^{d,l}$.
	
	Let $\v{m}= (m_1, \v{m}')$ and $\v{n}=(n_1,\v{n}')$ for fixed $(\v{m}', \v{n}') \in S_N^{2(l-1)}$ and define 
	\[ T(m_1,n_1)\ =\ T_{\v{m}',\v{n}'}(m_1, n_1)\ =\  Z_{d,l}(b)((m_1, \v{m}'),(n_1,\v{n}')).\]
	Then, $T$ satisfies 
	\begin{eqnarray*}
		T(m_1+N,n_1)&=& Z_{d,l}(b)(\v{m}+N\v{e}_1, \v{n})\ =\  e^{2\pi i \frac{n_1}{N}} T(m_1,n_1),\\
		T(m_1,n_1+N)&=& Z_{d,l}(b)(\v{m}, \v{n}+ N\v{e}_1)\  =\  T(m_1, n_1),
	\end{eqnarray*}
	so $T$ is $N$-quasiperiodic on $\Z_d^{2}$ (see equation \ref{eqn:quasiperiodic}).  By the unitary property of the Zak transform (Proposition \ref{prop:ZakProperties}), there exists a $b_1\in \ell_2^d$ so that $T=Z_{d,1}(b_1)$, and since $A \le |T(m_1,n_1)|^2 \le B$ for any $(m_1, n_1)\in \Z_d^2$, the same property shows that $G_d(b_1)$ is a Riesz basis for $\ell_2^d$ with bounds $A$ and $B$.  Thus, Theorem \ref{thm:betabound1d} shows that 

\begin{align*} C_{AB} \log(N) \le\ \sum_{(m_1,n_1)\in S_N^{2}} |\Delta_1 T_{\v{m}',\v{n}'}(m_1,n_1)|^2 + \sum_{(m_1,n_1)\in S_N^{2}} |\Gamma_1 T_{\v{m}',\v{n}'}(m_1,n_1)|^2 .
	 \end{align*}

	Since the choice of $(\v{m}', \v{n}') \in S_N^{2(l-1)}$ was arbitrary, this bound holds for any such choice.
	
	Thus, computing $\beta_1(b)$, we find

	\begin{gather*}
 \frac{1}{N^{2(l-1)}}  \sum\limits_{(\v{m}', \v{n}') \in S_N^{2(l-1)}} [ \sum_{(m_1,n_1)\in S_N^{2}} |\Delta T_{\v{m}',\v{n}'}(m_1,n_1)|^2 + \sum\limits_{(m_1,n_1)\in S_N^{2}} |\Gamma T_{\v{m}',\v{n}'}(m_1,n_1)|^2 ] \\
		\ge C_{AB} \log(N),
	\end{gather*}
	since the bound holds for each term inside the brackets, and $\beta_1(b)$ is simply an average of these terms.  Taking an infimum over all acceptable $b \in \ell_2^{d,l}$ proves the lower bound.

	\textbf{Upper bound:} To prove the upper bound, we adapt the construction used to prove the one-dimensional upper bound in \cite{Nitzan} to higher dimensions.  The sequence used in this construction builds on a continuous construction first given in \cite{BCGP}. Note that it suffices to prove the result for orthonormal bases, as the result for Riesz bases follows by scaling the constructed generator by the Riesz basis bounds.  
	
	In Section 4.3 of \cite{Nitzan}, it is shown that there is a constant $C>0$ such that for any $N\ge 2$, there exists a $b \in \ell_2^d$ such that $G_d(b)$ is an orthonormal basis for $\ell_2^d$ and 
	\[ \beta(b)\ =\  \sum_{(m,n)\in S_N^2} \left|\Delta Z_{d,1}(b)(m,n) \right|^2+\sum_{(m,n)\in S_N^2} \left|\Gamma Z_{d,1}(b)(m,n) \right|^2 \le C \log (N).\]
	
	For $\v{j} \in \Z_d^l$, let $b_l(\v{j})= b(j_1) b(j_2)\cdots b(j_l)$.  Then,
	\[ Z_{d,l}(b_l)(\v{m},\v{n})\ =\  Z_{d,1}(b)(m_1, n_1) \cdots Z_{d,1}(b)(m_l,n_l).\]
	Since $G_d(b)$ is an orthonormal basis for $\ell_2^d$, $Z_{d,l}(b_{l})$ is unimodular, and therefore, $G_{d,l}(b_l)$ is an orthonormal basis for $\ell_2^{d,l}$ by Proposition \ref{prop:ZakProperties}. We have, $\beta_{1}(b_l)$ is equal to
	
	\begin{gather*} \frac{1}{N^{2(l-1)}} \sum_{(\v{m}',\v{n}')\in \Z_N^{2(l-1)}} \left[\sum_{(m_1,n_1)\in S_N^2}  \left|\Delta Z_{d,1}(b)(m_1,n_1) \right| ^2 + \sum_{(m_1,n_1)\in S_N^2} \left|\Gamma Z_{d,1}(b)(m_1,n_1) \right| ^2\right] \\
 \le C\log(N). \end{gather*}
\end{proof}

Theorem \ref{thm:FiniteBLTHD} follows by combining Theorem \ref{thm:betaboundhd} with Proposition \ref{prop:alphabeta}.

%

\section{Proof of Theorem \ref{thm:FiniteQuantBLTHD}} \label{sec:proofFiniteQuantBLTHD}

In establishing a Finite Quantitative BLT for several variables, we follow a similar argument used to prove the one variable version (from \cite{Nitzan}), but there are some necessary updates to certain parts of the proof.  We include the details here for completeness.  

We start with a straightforward bound on the `smoothness' of $Z_{d,l}(b\ast \phi)$. This observation is analogous to Lemma 2.6 of \cite{Nitzan}. Let $\|\phi\|_{\ell_1^{d,l}} = \frac{1}{N^l} \sum_{\v{j}\in \Z_d^l} |\phi(\v{j})|$, and for $a, b \in \ell_2^{d,l}$, recall that $(a \ast b)(\v{k})= \frac{1}{N^l} \sum_{\v{j} \in \Z_d^l} a(\v{k}-\v{j})b(\v{j})$.   
\begin{lemma}\label{lem:convboundzd}
	Suppose $b, \phi \in \ell_2^{d,l}$ are such that $|Z_{d,l}(b)|^2 \le B$ everywhere.  Then, for any integer $t$, 
	\begin{eqnarray*}
		|Z_{d,l}(b\ast \phi) (\v{m}+t\v{e}_k,\v{n})-Z_{d,l}(b\ast \phi) (\v{m}, \v{n})| \le \frac{\sqrt{B} |t|}{N} \|N \Delta_k \phi\|_{\ell_1^{d,l}}.
	\end{eqnarray*}
\end{lemma}

\begin{proof}
	From Proposition \ref{prop:ZakProperties}, we have 
	\begin{eqnarray*}
		Z_{d,l}(b\ast \phi)(\v{m}, \v{n}) = \frac{1}{N^l} \sum_{\v{j} \in \Z_d^l} \phi(\v{j}) Z_{d,l}(b)(\v{m}-\v{j}, \v{n})= Z_{d,l}(b) \ast_1 \phi(\v{m},\v{n}).
	\end{eqnarray*}
	Therefore, we have 
	\begin{eqnarray*}
		& & |Z_{d,l}(b\ast \phi) (\v{m}+t\v{e}_k,\v{n})-Z_{d,l}(b\ast \phi) (\v{m}, \v{n})|\\
		&\le& \sum_{s=0}^{t-1} |Z_{d,l}(b\ast \phi) (\v{m}+(s+1)\v{e}_k,\v{n})-Z_{d,l}(b\ast \phi) (\v{m}+s\v{e}_k, \v{n})|\\
		&=& \sum_{s=0}^{t-1} \left| \frac{1}{N^l} \sum_{\v{j}\in \Z_d^l} Z_{d,l}(b)(\v{j}, \v{n}) [ \phi (\v{m}+(s+1)\v{e}_k -\v{j}) -\phi (\v{m}+s\v{e}_k -\v{j})] \right|\\
		&\le& \sum_{s=0}^{t-1} \frac{\sqrt{B}}{N^l} \sum_{\v{j}\in \Z_d^l} |\Delta_k \phi (\v{m}+s\v{e}_k - \v{j})|\ =\  \frac{\sqrt{B}}{N} t \|N\Delta_k \phi\|_{\ell_1^{d,l}}.
	\end{eqnarray*} 
\end{proof}

Next we extend the following Lemma 5.2 of \cite{Nitzan} to higher dimensions. The adjustments to this lemma for the higher dimensional setting are minimal, however we state the one-dimensional and multi-variable versions separately for comparison.

\begin{lemma}[Lemma 5.2, \cite{Nitzan}]\label{conv-lemma}
Let $A, B>0$ and $N\geq    200 \sqrt{B/A}$. There exist positive constants $\delta=\delta(A)$ and $C=C(A,B)$ such that   the following holds (with $d=N^2$).  Let
\begin{itemize}
\item[(i)]   \quad $Q, R \in \Z$ such that $1\leq Q,R    \leq (N/16) \cdot  \sqrt{A/B}$,
\item[(ii)]   \quad $\phi,\psi \in \ell_2^d$ such that $\sum_n|\Delta\phi(n)|\leq  10   R$ and $\sum_n|\Delta\psi(n)|\leq  10 Q$,
\item[(iii)] \quad $b\in \ell_2^d$ such that $A \leq |Z_d (b)|^2 \leq B$.
\end{itemize}
Then, there exists a set $S\subset ([0,N-1]\cap\Z)^2$ of size $|S|\geq C N^2/ QR$
 such that all $(u,v)\in S$ satisfy either
\begin{align}\label{conv-ineq-1}
|Z_d(b)(u,v)-Z_d(b\ast \phi)(u,v)|\geq\delta, \qquad \text{or} \\[2mm]
\label{conv-ineq-2}
|Z_d( \mathcal{F}_d  b)(u,v)-Z_d(( \mathcal{F}_d b)\ast \psi)(u,v)|\geq\delta.
\end{align}
\end{lemma}

\begin{lemma}\label{lem:generatorConvSmooth}
	Let $A,B>0$, $1\le k \le l$, and $N\ge  200\sqrt{B/A}$.  There exist positive constants $\delta=\delta(A)$ and $C=C(A, B)$, such that the following holds.  Let 
	\begin{enumerate}
		\item[(i)] $Q, R \in \Z$ be such that $1\le Q, R \le \frac{N}{16} \sqrt{\frac{A}{B}}$
		\item[(ii)] $\phi, \psi \in \ell_2^{d,l}$ be such that $\|N\Delta_k \phi\|_{\ell_1^{d,l}} \le 10 R$ and $\|N\Delta_k \psi\|_{\ell_1^{d,l}} \le 10 Q$
		\item[(iii)] $b \in \ell_2^{d,l}$ be such that $A\le |Z_{d,l}(b)|^2 \le B$.  
	\end{enumerate}
	Then, there exists a set $S\subset ([0,N-1]\cap \Z)^{2l}$ of size $|S|\ge CN^{2l}/Q R$ such that all $(\v{u}, \v{v}) \in S$ satisfy either 
	\begin{eqnarray}
	|Z_{d,l}(b) (\v{u}, \v{v}) - Z_{d,l}(b\ast \phi) (\v{u}, \v{v}) |&\ge& \delta, \text{     or}\label{eqn:first} \\
	|Z_{d,l}(\mathcal{F}_{d,l}b) (\v{u}, \v{v}) - Z_{d,l}((\mathcal{F}_{d,l}b)\ast \psi) (\v{u}, \v{v}) |&\ge& \delta. \label{eqn:second}
	\end{eqnarray}
\end{lemma}

\begin{proof}		
	Without loss of generality, we prove this for $k=1$.  
	
	As in Lemma 5.2 of \cite{Nitzan}, let $\delta_1= 2\sqrt{A} \sin(\pi ( \frac{1}{4}-\frac{1}{200}))$.  Also, choose $K$ and $L$ to be the smallest integers satisfying 
	\[ \frac{200  \sqrt{B} R}{9\delta_1}\le K \le N\ \ \ \  \text{and}\ \ \ \  \frac{\sqrt{B}}{\delta_1} \max\left\{ \frac{200 Q}{9}, 80 \pi\right\} \le L \le N.\]
	
	For $s, t \in \Z$, let 
	\[ \sigma_s=\left[ \frac{sN}{K}\right],\ \ \ \ \ \text{and} \ \ \ \ \ \omega_t=\left[ \frac{tN}{L}\right], \]
	and let $\Sigma=\inf_s \{ \sigma_{s+1}-\sigma_s\}\ge \left[\frac{N}{K}\right] \ge \frac{N}{K}$, $\Omega = \inf_t\{ \omega_{t+1}-\omega_t\}\ge \frac{N}{L}$.  Then, we have 
	\[ \Sigma \Omega \ge C_1 \frac{N^2}{QR},\]
	where $C_1$ can be chosen to be 
	\[ C_1= \left[(\frac{200 \sqrt{B}}{9\delta_1}+1)(\frac{\sqrt{B}}{\delta_1} \max(\frac{200}{9},80\pi) +1)\right]^{-1}.\]
	
	We recall the following definition from \cite{Nitzan}.  For  $(u,v) \in ([0,\Sigma-1] \cap \Z) \times ([0,\Omega-1]\cap \Z)$, let 
	\[ \text{Lat}(u,v)=\{(u+\sigma_s,v+\omega_t):(s,t)\in ([0,K-1]\cap \Z)\times ([0,L-1]\cap \Z)\},\]
	and 
	\[ \text{Lat}^*(u,v)=\{(N-v-\omega_t,u+\sigma_s):(s,t)\in ([0,K-1]\cap \Z)\times ([0,L-1]\cap \Z)\}.\]
	Note that $\text{Lat}(u,v)$ and $\text{Lat}(u',v')$ are disjoint for distinct $(u,v)$ and $(u',v')$, and similarly for $\text{Lat}^*(u,v)$.  However, it is possible that $\text{Lat}(u,v) \cap \text{Lat}^*(u',v') \neq \emptyset$ for some $(u,v)$ and $(u',v')$.
	
	Now similarly, for any $(\v{m}', \v{n}') \in ([0,N-1]\cap \Z)^{2(l-1)}$, let
	\[ 
	\text{Lat}_{(\v{m}',\v{n}')}(u,v)= \{ ((m_1,\v{m}'),(n_1,\v{n}')): (m_1,n_1)\in \text{Lat}(u,v)\},\] 
	and
	\begin{eqnarray*}
		\text{Lat}^*_{(\v{m}',\v{n}')}(u,v)= \{ ((n_1,N-\v{n}'),(m_1,\v{m}')):(n_1,m_1)\in \text{Lat}^*(u,v) \}.
	\end{eqnarray*}
	Here, by $N-\v{n'}$ we mean $(N-n'_1, N-n'_2,..., N-n'_{l-1})$.  We have that $\text{Lat}_{(\v{m}',\v{n}')}(u,v) \cap \text{Lat}_{(\v{m}'',\v{n}'')}(u',v')=\emptyset$ unless it holds that  $((u,\v{m}'),(v,\v{n}'))=$ $ ((u',\v{m}''),(v',\v{n}''))$, and similar properties for $\text{Lat}^*_{(\v{m}',\v{n}')}(u,v)$.
	
	Now, fix $(\v{m}', \v{n}') \in ([0,N-1]\cap \Z)^{2(l-1)}$, and consider \[T(m_1, n_1)=T_{\v{m}',\v{n}'}(m_1,n_1)=Z_{d,l}(b)( (m_1,\v{m}'), (n_1, \v{n}')),\] for $(m_1, n_1) \in \Z_d^2$.  Note that $T$ is $N$-quasiperiodic on $\Z_d^2$, and satisfies $A\le |T|^2 \le B$.  
	
	For each $(u,v)\in([0,\Sigma-1] \cap \Z) \times ([0,\Omega-1]\cap \Z)$, Corollary 3.6 of \cite{Nitzan} guarantees at least one point $(s,t) \in ([0,K-1] \cap \Z) \times ([0,L-1]\cap \Z)$ so that either 
	\begin{eqnarray}
	&|T(u+\sigma_{s+1}, v+\omega_t) - T(u+\sigma_s,v+\omega_t)| \ge \delta_1, \text{ or} \label{eqn:third}\\
	&|T(u+\sigma_s, v+\omega_{t+1}) - T(u+\sigma_s,v+\omega_t)| \ge \delta_1. \label{eqn:fourth}
	\end{eqnarray}

We now make a claim which will furnish the last part of the proof of the lemma.
	
	\begin{claim}
		For $u$, $v$, $\sigma_s$, $\omega_t$, $\v{m}'$ and $\v{n}'$ as above,  
		\begin{enumerate}
			\item[(i)] If \eqref{eqn:third} is satisfied, then there exists $(\v{a},\v{b})\in \text{Lat}_{(\v{m}',\v{n}')}(u,v)$ so that \eqref{eqn:first} is satisfied for $\delta=\frac{\delta_1}{20}$. 
			\item[(ii)] If \eqref{eqn:fourth} is satisfied, then there exists $(\v{a},\v{b})\in \text{Lat}^*_{(\v{m}',\v{n}')}(u,v)$ so that \eqref{eqn:second} is satisfied for $\delta= \frac{\delta_1}{40}$
		\end{enumerate}
	\end{claim}
	
	Before proving this claim, we show how to complete the proof of the lemma.  For a fixed $(\v{m}', \v{n}')$ there are $\Sigma \Omega \ge C_1 \frac{N^2}{QR}$ distinct choices of $(u,v)$ to consider and each of them either falls in part  (\textit{i}) or (\textit{ii}) of the claim.  Let $S^1_{(\v{m}',\v{n}')}$ be the set of $(u,v)$ points which fall into category (\textit{i}), and similarly let $S^2_{(\v{m}', \v{n}')}$ be the set of $(u,v)$ points which fall into category (\textit{ii}).  Then, for either $i=1,2$, we must have 
	\begin{equation}\label{eqn:Sbound} 
	|S^i_{(\v{m}', \v{n}')} |\ \ge\  \frac{C_1N^2}{2QR}.
	\end{equation}
	
	Now, there are $N^{2(l-2)}$ possible choices of $(\v{m}', \v{n}')$.  Let $S_1$ be the set of all $(\v{m}', \v{n}')$ such that \eqref{eqn:Sbound} is satisfies with $i=1$, and let $S_2$ be the set of all $(\v{m}', \v{n}')$ such that \eqref{eqn:Sbound} is satisfied with $i=2$.  So at least one of $S_1$ or $S_2$ must contain $N^{2(l-2)}/2$ elements.  
	
	In the case that $S_1$ contains this many elements (the $S_2$ case is nearly identical and left to the reader), since $\text{Lat}_{(\v{m}', \v{n}')}(u,v)$ are disjoint for distinct $((u,\v{m}'),(v,\v{n}'))$, we find at least $\frac{C_1N^{2l}}{4QR}= C\frac{N^{2l}}{QR}$ distinct points all satisfy \eqref{eqn:first} if $i=1$.  The lemma is then proved conditioning on the claim above. We then establish finally the two part claim. \\

	\textit{Proof of Claim.}
	For both parts we use properties of the Zak transform detailed in Proposition \ref{prop:ZakProperties}. First we show part \textit{i)}.  Let $H(u,v)=Z_{d,l}(b\ast \phi)((u,\v{m}'),(v,\v{n}'))$.  Note that Lemma \ref{lem:convboundzd} and the assumptions on $R$ and $\|N \Delta_1 \phi\|_{\ell_1^{d,l}}$  imply that for any integer $t$ satisfying $t\le \frac{2N}{K}$, 
	\begin{eqnarray}
	|H(u+t,v)-H(u,v)|&\le \frac{2\sqrt{B} }{K} \|N \Delta_1 \phi\|_{\ell_1^{d,l}}\le \frac{ 20\sqrt{B} R}{K}\le \frac{9\delta_1}{10}. \label{eqn:Hbound}
	\end{eqnarray}
	So, if \eqref{eqn:third} is satisfied, using \eqref{eqn:Hbound}, we have
	\begin{eqnarray*}
		\delta_1 &\le& |T(u+\sigma_{s+1}, v+\omega_t) - T(u+\sigma_s,v+\omega_t)|\\
		&\le& |T(u+\sigma_{s+1}, v+\omega_t) - H(u+\sigma_{s+1}, v+\omega_t)| \\
		& &\ \ \ \ \ + \frac{9\delta_1}{10}+|T(u+\sigma_{s}, v+\omega_t) - H(u+\sigma_{s}, v+\omega_t)|.
	\end{eqnarray*}
	Upon rearranging terms, we find 
	\begin{eqnarray*}
		\frac{\delta_1}{10} &\le |T(u+\sigma_{s+1}, v+\omega_t) - H(u+\sigma_{s+1}, v+\omega_t)|+|T(u+\sigma_{s}, v+\omega_t) - H(u+\sigma_{s}, v+\omega_t)|,
	\end{eqnarray*}
	which shows that \eqref{eqn:first} is satisfied for $\delta'= \frac{\delta_1}{20}$, and for either $((u+\sigma_{s+1},\v{m}'),(v+\omega_t,\v{n}'))$ or $((u+\sigma_{s},\v{m}'),(v+\omega_t,\v{n}'))$. If $(u+\sigma_{s+1}, v+\omega_t)$ is not in $\text{Lat}(u,v)$, by the N-quasiperiodicity of $T$, we may find another point in $\text{Lat}(u,v)$ which satisfies the same bound.  
	
	Now we prove part \textit{ii)}.  Letting $\hat{b}= \mathcal{F}_{d,l}(b)$, we have,
	\begin{eqnarray*}
		\delta_1&\le& |T(u+\sigma_s, v+\omega_{t+1}) - T(u+\sigma_s,v+\omega_t)| \\
		&=& |Z_{d,l}(b)((u+\sigma_s, \v{m}'),(v+\omega_{t+1},\v{n}')) - Z_{d,l}(b)((u+\sigma_s, \v{m}'),(v+\omega_{t},\v{n}'))|\\
		&=& |Z_{d,l}(\hat{b})((-v-\omega_{t+1},-\v{n}'),(u+\sigma_s, \v{m}'))\\
		&\ &\  - e^{-2\pi i (\omega_{t+1}-\omega_t)(u+\sigma_s) /d} Z_{d,l}(\hat{b})((-v-\omega_{t},-\v{n}'),(u+\sigma_s, \v{m}'))|\\
		&=& |Z_{d,l}(\hat{b})((N-v-\omega_{t+1},N-\v{n}'),(u+\sigma_s, \v{m}'))\\
		&\ &\ - e^{-2\pi i (\omega_{t+1}-\omega_t)(u+\sigma_s) /d} Z_{d,l}(\hat{b})((N-v-\omega_{t},N-\v{n}'),(u+\sigma_s, \v{m}'))|,
	\end{eqnarray*}
	where we have used that $Z_{d,l}(b)(\v{m},\v{n})= e^{2\pi i \v{m}\cdot \v{n}/d} Z_{d,l}(\widehat{b})(-\v{n},\v{m})$ in the second step, and for the last step we have used $N$-quasiperiodicity.  
	
	Let $\widetilde{T}(v,u)=Z_{d,l}(\hat{b})((v,N-\v{n}'),(u, \v{m}'))$, 
	and $\widetilde{H}(v,u)= Z_{d,l}(\hat{b}\ast \psi)((v,N-\v{n}'),(u, \v{m}'))$.
	Then, 
	\begin{eqnarray*}
		\delta_1 & \le & |\widetilde{T}(N-v-\omega_{t+1},u+\sigma_s)- e^{-2\pi i (\omega_{t+1}-\omega_t)(u+\sigma_s) /d}\widetilde{T}(N-v-\omega_{t},u+\sigma_s)|\\
		& \le & | \widetilde{T}(N-v-\omega_{t+1},u+\sigma_s)- \widetilde{T}(N-v-\omega_{t},u+\sigma_s)| +\frac{\delta_1}{20}. 
	\end{eqnarray*}
	Combining these, we see that 
	\begin{eqnarray*}
		\frac{19}{20} \delta_1 &\le& | \widetilde{T}(N-v-\omega_{t+1},u+\sigma_s)- \widetilde{T}(N-v-\omega_{t},u+\sigma_s)|.  
	\end{eqnarray*}
	Arguing as in the first case above, and replacing $H$ by $\widetilde{H}$ and $T$ by $\widetilde{T}$, we find that either $((N-v-\omega_{t+1},N-\v{n}'),(u, \v{m}'))$, or $((N-v-\omega_{t},N-\v{n}'),(u, \v{m}'))$ satisfy \eqref{eqn:second}, with $\delta= \frac{\delta_1}{40}$.  Again, using quasi-periodicity, we can guarantee that there is a point in $\text{Lat}^*_{(\v{m}',\v{n}')}(u,v)$ satisfying \eqref{eqn:second}.  
	
\end{proof} \vspace{2 mm}

Finally, we follow the construction of \cite{Nitzan} to create the functions $\phi$ and $\psi$ appearing in the previous lemma (Lemma \ref{lem:generatorConvSmooth}) which in turn are used to prove Theorem \ref{thm:FiniteQuantBLTHD}. Let $\rho$ : $\mathbb{R}\rightarrow \mathbb{R}$ be the inverse Fourier transform of 
\[\hat{\rho}(\xi)= \begin{dcases} \hspace{15 mm} 1,\hspace{10 mm}  |\xi|\leq 1/2 \\ 2(1-\xi sgn(\xi)),\hspace{5 mm}  1/2\leq |\xi|\leq 1 \\ \hspace{15 mm} 0,\hspace{10 mm} |\xi|\geq1\end{dcases}. \] 

For $f \in L^2(\R)$ satisfying $\sup_{t \in \R} |t^2 f(t)|<\infty$ and $\sup_{\xi \in \R} |\xi^2 \widehat{f}(t)|<\infty$, let 
\[ P_N f(t)\ =\  \sum_{k=-\infty}^\infty f(t+kN)\]
and for an $N$-periodic continuous function $h$, let 
\[ S_N h\ =\ \{ h(j/N)\}_{j =0}^{d-1}.\]

Let $\rho_R(t)=R\rho(Rt)$.  Fix $1\le k \le l$, and for $\v{j} \in I_d^l$ define the vector  $\v{j}'=(j_1,..., j_{k-1},j_{k+1},...,j_l) \in I_{d}^{l-1}$, and let 
\[ \phi_{R,k}(\v{j}) \ =\  N^{l-1} \delta_{\v{j}', \v{0}}  \left(S_N P_N \rho_{R} (j_k) \right).\]
Now $\phi_{R,k} (\v{j})$ is equal to $ \left(S_N P_N \rho_{R} (j_k) \right)$ when $j_i=0$ for each $i \neq k$, and is zero otherwise.  
\begin{lemma}
	Let $\phi_{R,k}$ be as above for a positive integer $R$.  Then,
	\[ \|N \Delta_k \phi_{R,k}\|_{\ell_1^{d,l}} \ \le\   10 R.\]
\end{lemma}
\begin{proof}
	We have 
	\begin{eqnarray*}
		\|N \Delta_k \phi_{R,k}\|_{\ell_1^{d,l}}&=& \frac{1}{N^l} \sum_{\v{j} \in I_d^l} N| \Delta_k \phi_{R,k} (\v{j})| \ =\ \sum_{j_k \in I_d} |\Delta S_N P_N \rho_{R}(j_k)|.
	\end{eqnarray*}
	Lemma 2.10 and Lemma 5.1 of  \cite{Nitzan} show that the right hand side is bounded by $10R$. 
\end{proof}

We now have sufficient tools to prove the Finite Quantitative BLT, Theorem \ref{thm:FiniteQuantBLTHD}.

\begin{proof}[Theorem \ref{thm:FiniteQuantBLTHD}]
	For simplicity we show the result for $k=1$.  
	Let $R$ and $Q$ be integers such that $1 \le R, Q \le (N/16)\sqrt{A/B}$.  Let $\phi= \phi_{R,1}$ and $\psi=\phi_{Q,1}$, and note that Lemma \ref{lem:convboundzd} shows that 
	\[ \|N\Delta_1\phi \|_{\ell_1^{d,l}} \le 10 R, \text{ and } \|N\Delta_1\psi \|_{\ell_1^{d,l}} \le 10 Q.\]
	Proposition 2.8 of \cite{Nitzan}, and the fact that $\mathcal{F}_d(N\delta_{j,0})(k)=1$ for all $k \in I_d$, shows that 
	\begin{align}
	\mathcal{F}_{d,l} (\phi)(\vec{k}) &= \mathcal{F}_d(S_N P_N \rho_{R})(k_1) \nonumber \\
		& = (S_N P_N \mathcal{F}(\rho_{R})) (k_1) = (S_N P_N \widehat{\rho}(\cdot/R)) (k_1),\label{eqn:phifunction}
	\end{align}
	and since $R<N/2$, then $0\le \widehat{\phi}=\mathcal{F}_{d,l} (\phi)\le 1$. Also, $\widehat{\phi}(\v{k}) = 1$ for any $\v{k}$ which satisfies $k_1 \in [-RN/2, RN/2]$, independent of the values of $k_2, ..., k_l$, that is, for any $\v{k} \in S_{NR,1}$.  The same holds for $\widehat{\psi}=\mathcal{F}_{d,l}(\psi)$ with $Q$ replacing $R$.  
	Applying Lemma \ref{lem:generatorConvSmooth}, we find a constant $C$ such that 
	\begin{eqnarray*}
		\frac{CN^{2l}}{Q R} &\le& \sum_{(\v{m}, \v{n})\in \ell_2(S_N^{2l})} |Z_{d,l}(b)(\v{m},\v{n})-Z_{d,l}(b\ast \phi)(\v{m},\v{n})|^2\\		
		&\ & +\sum_{(\v{m}, \v{n})\in \ell_2(S_N^{2l})} |Z_{d,l}(\widehat{b})(\v{m},\v{n})-Z_{d,l}(\widehat{b}\ast \psi)(\v{m},\v{n})|^2,
	\end{eqnarray*}
	where here we have let $\widehat{b}=\mathcal{F}_{d,l}(b)$.  Using that $Z_{d,l}$ and $\mathcal{F}_{d,l}$ are both isometries and the properties of $\phi$ and $\psi$ listed above, we have
	\begin{eqnarray*}
		\frac{C}{Q R} &\le& \|Z_{d,l}(b)-Z_{d,l}(b\ast \phi)\|_{\ell_2(S_N^{2l})}^2+ \|Z_{d,l}(\widehat{b})-Z_{d,l}(\widehat{b}\ast \psi)\|_{\ell_2(S_N^{2l})}^2\\
		&=& \|b-b\ast \phi\|_{\ell_2^{d,l}}^2 + \|\widehat{b}- \widehat{b}\ast \psi\|_{\ell_2^{d,l}}^2\\
		&=& \| \widehat{b}( 1-\widehat{\phi})\|_{\ell_2^{d,l}}^2 +\| b(1-\widehat{\psi})\|_{ \ell_2^{d,l}}^2\\
		&\le& \frac{1}{N^l}\sum\limits_{|j_1|\ge \frac{NR}{2}}|\mathcal{F}_{d}b(\v{j})|^{2}+\frac{1}{N^l}\sum\limits_{|j_1|\ge \frac{NQ}{2}}|b(\v{j})|^{2}.
	\end{eqnarray*} 
\end{proof}

\section{Nonsymmetric Finite BLT and Applications of the Quantitative BLTs} \label{FQBLTHDapps}

In this penultimate section, we prove the nonsymmetric finite BLT, Theorem \ref{thm:NonsymFiniteBLT}, and the uncertainty principles of Theorem \ref{thm:QuantBLTCorollaryNonsymmetric}. We show each of these follows from a version of the Quantitative BLT, however, the details of the proof of Theorem \ref{thm:NonsymFiniteBLT} are more difficult due to subtleties from discreteness. For this reason, we first prove Theorem \ref{thm:QuantBLTCorollaryNonsymmetric} which shows the central idea of both proofs without the added technical difficulty.

First, we state the higher dimensional quantitative BLT of \cite{Temur}.  For notational simplicity, we write $\{|x_k|\ge s\}$ to mean $\{ x\in \R^l: |x_k|\ge s\}$ in situations where the dependence on $l$ is clear. 

\begin{theorem}[Theorem 1, \cite{Temur}]\label{thm:ContQuantBLTHD}
	Let $g \in L^2(\R^l)$ be such that the Gabor system generated by $g$
	\begin{eqnarray*} G(g) &=& \{ e^{2\pi i n\cdot x} g(x-m)\}_{(m,n)\in \Z^{2l}} \end{eqnarray*}
	is a Riesz basis for $L^2(\R^l)$.  Let $R, Q \ge 1$ be real numbers.  Then, there is a constant $C$ which only depends on the Riesz basis bounds of $G(g)$ such that for any $1 \le k \le l$
	\begin{equation} \label{eqn:QuantBLTConclus}
	\int_{|x_k|\ge R} |g(x)|^2 dx + \int_{|\xi_k|\ge Q} |\widehat{g}(\xi)|^2 d\xi \ \ge\  \frac{C}{RQ}.
	\end{equation}
\end{theorem}

\begin{remark}
	In \cite{Temur}, the conclusion of this theorem is stated where the integrals in \eqref{eqn:QuantBLTConclus} are taken over $\R^l \setminus \mathcal{R}$ and $\R^l \setminus \mathcal{Q}$, respectively, where $\mathcal{Q}$ and $\mathcal{R}$ are finite volume rectangles in $\R^d$.  However, a straightforward limiting argument shows that the result holds after removing `infinite volume' rectangles, as in the statement above.   
\end{remark}

\begin{proof}[Theorem \ref{thm:QuantBLTCorollaryNonsymmetric}]
	We will prove this for $k=1$ without loss of generality.  
	
	Let $1\le S <\infty$, and choose $R=S^{1/p}$ and $Q=S^{1/q}$.  Note $1\le R,Q<\infty$ for any value of $S$. Theorem \ref{thm:ContQuantBLTHD} then shows that for $C$ only depending on the Riesz basis bounds of $G(f)$, 
	\begin{eqnarray} \label{eqn:QuantBLTwithS}
	\frac{C}{S^\tau}&=&\frac{C}{S^{\frac{1}{p}+\frac{1}{q}}} \ \le\  \int_{|x_1|\ge S^{1/p}} |g(x)|^2 dx + \int_{|\xi_1|\ge S^{1/q}} |\widehat{g}(\xi)|^2 d\xi.
	\end{eqnarray}
	In each case, the result follows by integrating both sides of \eqref{eqn:QuantBLTwithS} over a particular set of $S$ values, and then using Tonelli's Theorem to interchange the order of integration.  
	
	\textbf{Case 1: $\tau=\frac{1}{p}+\frac{1}{q} <1$}. We have, 

	\begin{eqnarray*}
\hspace{-10 mm} C\frac{(1-2^{\tau-1})}{1-\tau}T^{1-\tau}&=&C\int_1^T S^{-\tau} dS\\
		&\le& \int_{\R^{l-1}} \int_0^T \int_{|x_1|\ge S^{1/p}} |g(x_1,x')|^2 dx_1 dS dx' \\ &+& \int_{\R^{l-1}} \int_0^T \int_{|\xi_1|\ge S^{1/q}} |\widehat{g}(\xi_1,\xi')|^2 d\xi_1 dS d\xi'\\
		&\le& \int_{\R^{l}} \int_0^{\min(|x_1|^p,T)}  |g(x)|^2 dS dx+ \int_{\R^{l}} \int_0^{\min(|\xi_1|^q,T)} |\widehat{g}(\xi)|^2 dS d\xi\\
		&=& \int_{\R^{l}} \min(|x_1|^p,T)  |g(x)|^2  dx+ \int_{\R^{l}} \min(|\xi_1|^q,T)|\widehat{g}(\xi)|^2  d\xi.
	\end{eqnarray*}

	\textbf{Case 2: $\tau=\frac{1}{p}+\frac{1}{q} =1$}.  Similarly, we have 

	\begin{eqnarray*}
		C \log{T}&=&C\int_1^T S^{-1} dS\\
		&\le& \int_{\R^{l-1}} \int_0^T \int\limits_{|x_1|\ge S^{1/p}} |g(x_1,x')|^2 dx_1 dS dx' + \int_{\R^{l-1}} \int_0^T \int\limits_{|\xi_1|\ge S^{1/q}} |\widehat{g}(\xi_1,\xi')|^2 d\xi_1 dS d\xi'\\
		&\le& \int_{\R^{l}} \min(|x_1|^p,T)  |g(x)|^2  dx+ \int_{\R^{l}} \min(|\xi_1|^q,T)|\widehat{g}(\xi)|^2  d\xi.
	\end{eqnarray*}
	
	\textbf{Case 3: $\tau=\frac{1}{p}+\frac{1}{q} >1$}. Finally, in this case

	\begin{eqnarray*}
		\frac{C}{\tau-1}&=&C\int_1^\infty S^{-\tau} dS\\
		&\le& \int_{\R^{l-1}} \int_0^\infty \int\limits_{|x_1|\ge S^{1/p}} |g(x_1,x')|^2 dx_1 dS dx'+ \int_{\R^{l-1}} \int_0^\infty \int\limits_{|\xi_1|\ge S^{1/q}} |\widehat{g}(\xi_1,\xi')|^2 d\xi_1 dS d\xi'\\
		&=& \int_{\R^{l}} |x_1|^p |g(x)|^2  dx+ \int_{\R^{l}} |\xi_1|^q |\widehat{g}(\xi)|^2  d\xi.
	\end{eqnarray*}

\end{proof}

The following result generalizes part (ii) of Theorem \ref{thm:BLT}.  
\begin{theorem}\label{thm:compactFunctionQuantCorr}
	Suppose $1\le p < \infty$, and $g \in L^2(\R^l)$ is such that $G(g)= \{e^{2\pi i n \cdot x} g(x-m)\}_{(m,n) \in \Z^{2l}}$ is a Riesz basis for $L^2(\R^l)$ and $g$ is supported in $(-M,M)^l$.  Then, there exists a constant $C$ depending only on the Riesz basis bounds of $G(g)$ such that for any $1\le k \le 1$ and any $2 \le T \le \infty$ each of the below hold.
	\begin{enumerate}
		\item[(i)] If $p>1$, then 
		\[ \frac{C(1-2^{1/p-1})}{M(1-1/p)}\ \le\  \int_{\R^l} \min(|\xi_k|^p, T) |\widehat{g}(\xi)|^2 d\xi.\]
		\item[(ii)] If $p = 1$, then 
		\[ \frac{C\log(T)}{M}\ \le\  \int_{\R^l} \min(|\xi_k|, T) |\widehat{g}(\xi)|^2 d\xi.\]
		\item[(iii)] If $p<1$, then 
		\[ \frac{C}{M(1/p-1)}\  \le\  \int_{\R^l} |\xi_k|^p, |\widehat{g}(\xi)|^2 d\xi.\]
	\end{enumerate}
	This result also holds when $g$ and $\widehat{g}$ are interchanged.
\end{theorem}
The proof is nearly identical to that of Theorem \ref{thm:QuantBLTCorollaryNonsymmetric}, after noticing that by applying the quantitative BLT with $R=M$, the integral related to $|g(x)|^2$ is zero due to the support assumption.  Note that letting $T \rightarrow \infty$ in part (ii) gives part (ii) of Theorem \ref{thm:BLTHD}.

Finally, we focus on the finite nonsymmetric BLT. For $1\le p,q<\infty$ and $b \in \ell_2^{d,l}$, let 
\[ \alpha_k^{p,q}(b)\ =\  \frac{1}{N^{l}} \sum_{\v{j} \in \Z_{d}^l} \left|\frac{j_k}{N}\right|^p |b(\v{j})|^2+ \frac{1}{N^{l}} \sum_{\v{j} \in \Z_{d}^l} \left|\frac{j_k}{N}\right|^q |\mathcal{F}_{d,l}b(\v{j})|^2.\] To give a finite dimensional analog of part (ii) of Theorem \ref{thm:BLT}, it will be convenient to define $\alpha^{p,\infty}_{k}(b)$ and $\alpha^{\infty, q}_k(b)$ as
\[\alpha^{p,\infty}_{k}(b)\ =\  \frac{1}{N^{l}} \sum_{\v{j} \in \Z_{d}^l} \left|\frac{j_k}{N}\right|^p |b(\v{j})|^2, \ \ \alpha^{\infty,q}_k(b)\ =\  \frac{1}{N^{l}} \sum_{\v{j} \in \Z_{d}^{l}} \left|\frac{j_k}{N}\right|^q |\mathcal{F}_{d,l}b(\v{j})|^2.\]

\begin{theorem}\label{thm:NonsymFiniteBLTwithInfinity}
	Let $A,B>0$ and $1\le p,q< \infty$ and let $\tau=\frac{1}{p}+\frac{1}{q}$.  Assume $b\in \ell_2^{d,l}$ generates an $A, B$-Gabor Riesz basis for $\ell_2^{d,l}$.  There exists a constant $C>0$, depending only on $A, B, p$ and $q$ such that the following holds. Let $N \ge 200\sqrt{B/A}$. 
	\begin{enumerate}
		\item[(i)] If $\tau=\frac{1}{p}+\frac{1}{q}<1$,
		\begin{eqnarray*} C \frac{N^{1-\tau}}{1-\tau} &\le& \alpha^{p,q}_{k}(b).\end{eqnarray*}
		\item[(ii)] If $\tau=\frac{1}{p}+\frac{1}{q}=1$, 
		\begin{eqnarray*} C \log(N) &\le& \alpha^{p,q}_{k}(b.)\end{eqnarray*}
		\item[(iii)] If $\tau=\frac{1}{p}+\frac{1}{q}>1$,
		\begin{eqnarray*} C\frac{1-(200/16)^{1-\tau}}{\tau-1} &\le& \alpha^{p,q}_{k}(b).\end{eqnarray*}
	\end{enumerate}
	Also, if $\mathcal{F}_{d,l}(b)$ is supported in the set $(-\gamma_N N/2,\gamma_N N/2)\cap \Z$ where $\gamma_N= \lfloor (N/16)\sqrt{A/B}\rfloor$, then parts (i), (ii), and (iii) hold with $\tau=\frac{1}{p}$ and $\alpha^{p,q}(b)$ replaced by $\alpha^{p,\infty}(b)$.  Similarly, if $b$ is supported in the set $(-\gamma_N N/2,\gamma_N N/2)\cap \Z$ then parts (i), (ii), and (iii) hold with $\tau=\frac{1}{q}$ and $\alpha^{p,q}(b)$ replaced by $\alpha^{\infty, q}(b)$.  
\end{theorem}

We start with a lemma giving a bound on a typical sum arising in the proof which follows. Similar to above, $\{ b>|j_k| \ge a\}$ will be used to denote $\{ \v{j} \in I_d^l: b>|j_k| \ge a\}$.  
\begin{lemma}\label{lem:BasicSumBounds}
	Let $1\le \nu<\infty$, $N>200 \nu$, $c=1/(16\nu)$, and $\gamma_N=\lfloor c N\rfloor$. If $0<\alpha\le 1$, then for any $b \in \ell_2^{d,l}$, we have
	\[ \sum_{S=1}^{\gamma_N} \sum_{|j_k|\ge NS^\alpha/2} |b(\v{j})|^2  \ \le\  2^{1/\alpha} \sum_{\v{j} \in \Z^d} \left|\frac{j_k}{N}\right|^{1/\alpha} |b(\v{j})|^2,\]
	where $C_\alpha$ only depends on $\alpha$.  
\end{lemma}
Note, we will apply this lemma with $\nu= \sqrt{B/A}$ where $A$ and $B$ are Riesz basis bounds of $G_{d,l}(b)$ for some $b \in \ell_2^{d,l}$.  However, this lemma holds regardless of whether $G_{d,l}(b)$ is basis for $\ell_2^{d,l}$.  
\begin{proof}
	Rearranging terms, we have

	\begin{eqnarray}
	\sum_{S=1}^{\gamma_N} \sum_{|j_k|\ge NS^\alpha/2} |b(\v{j})|^2 = \sum_{m=1}^{\gamma_N-1} m \sum_{\frac{N(m+1)^\alpha}{2}> |j_k|\ge \frac{Nm^\alpha}{2}} |b(\v{j})|^2 +  \gamma_N \sum_{|j_k|\ge \frac{N \cdot \gamma_N^\alpha}{2}} |b(\v{j})|^2. \label{eqn:CantComeUpWithGoodName}
	\end{eqnarray}

	Note that for some $m$, if $j_k$ satisfies $|j_k|\ge \frac{Nm^\alpha}{2}$, then  $m \le 2^{1/\alpha} \left| \frac{j_k}{N}\right|^{1/\alpha}$.
	Then, from \eqref{eqn:CantComeUpWithGoodName}, we find
	\begin{eqnarray*}
		\sum_{S=1}^{\gamma_N} \sum_{|j_k|\ge NS^\alpha/2} |b(\v{j})|^2 &\le& 2^{1/\alpha}\sum_{m=1}^{\gamma_N-1} \sum_{\frac{N(m+1)^\alpha}{2}> |j_k|\ge \frac{Nm^\alpha}{2}}  \left| \frac{j_k}{N}\right|^{1/\alpha}|b(\v{j})|^2 \\ &\ +&    2^{1/\alpha} \sum_{|j_k|\ge \frac{N \gamma_N^\alpha}{2}}  \left| \frac{j_k}{N}\right|^{1/\alpha}|b(\v{j})|^2\\
		&\le& 2^{1/\alpha}\sum_{\v{j} \in I_d^l} \left |\frac{j_k}{N}\right|^{1/\alpha} |b(\v{j})|^2.
	\end{eqnarray*}  
\end{proof}

\begin{proof}[Theorem \ref{thm:NonsymFiniteBLTwithInfinity}]
	We prove the result for $k=1$.  We treat the case where $p$ and $q$ are both finite and the case where one of these is infinite separately. Below, we take $\tau=\frac{1}{p}+\frac{1}{q}$. 
	
	\textbf{Case 1: $1\le p, q <\infty$}. Let $S$ be an integer satisfying $1\le S \le \gamma_N$ where $\gamma_N = \lfloor (N/16)\sqrt{A/B}\rfloor$, and $R= \lceil S^{1/p}\rceil$, $Q=\lceil S^{1/q}\rceil$ if $1<p,q<\infty$, $R=S$ if $p=1$, and $Q=S$ if $q=1$.  Note that these choices force $1\le R, Q \le \gamma_N$.  Then, for a constant $C$ only depending on $A$ and $B$,  Theorem \ref{thm:FiniteQuantBLTHD} gives 
	\begin{eqnarray*}
		\frac{C/4}{S^{\tau}} & \le& \frac{C}{RQ} \ \le\  \frac{1}{N^l} \sum_{|j_k|\ge \frac{NS^{1/p}}{2}} |b(\v{j})|^2 + \frac{1}{N^l} \sum_{|j_k|\ge \frac{NS^{1/q}}{2}} |\mathcal{F}_{d,l}b (\v{j})|^2.  
	\end{eqnarray*}
	
	Summing over the values of $S$ in $\{1,..., \gamma_N\}$ and applying Lemma \ref{lem:BasicSumBounds} with $\tau=\sqrt{B/A}$, to find

	\begin{eqnarray*}
		\frac{C}{4} \sum_{S=1}^{\gamma_N} S^{-\tau} &\le& \frac{1}{N^l} \sum_{S=1}^{\gamma_N}\sum_{|j_k|\ge \frac{NS^{1/p}}{2}} |b(\v{j})|^2 + \frac{1}{N^l} \sum_{S=1}^{\gamma_N}\sum_{|j_k|\ge \frac{NS^{1/q}}{2}} |\mathcal{F}_{d,l}b (\v{j})|^2 \le  C' \alpha^{p.q}_k(b)
	\end{eqnarray*}

	where $C'$ is a constant only depending on $p$ and $q$.  Updating the constant $C$, (it now depends on $A$, $B$, $p$, $q$)
	\[ C \sum_{S=1}^{\gamma_N} S^{-\tau} \ \le\  \alpha^{p,q}_k(b,N,l).\]
	
	The proof of Case 1 follows by noting that 
	\begin{equation} \label{eqn:powersumbounds} \sum_{S=1}^{\gamma_N} S^{-\tau} \ \ge\   \begin{cases} 
	C_{\tau,A,B} \frac{N^{1-\tau}}{1-\tau} & 0<\tau < 1 \\
	
	C_{\tau,A,B} \log(N) & \tau=1\\
	
	\frac{(1-(200/16)^{1-\tau})}{1-\tau} & \tau>1
	\end{cases},
	\end{equation}
	where the constants $C_{\tau,A,B}$ depend only on $\tau$, $A$, and $B$.

	\textbf{Case 2: One of $p$ or $q$ is $\infty$}. We can assume without loss of generality that $q=\infty$ and $1\le p<\infty$.  With this in mind, assume $b$ generates an $A,B$-Gabor Riesz basis for $\ell_2^{d,l}$, and further suppose $\mathcal{F}_{d,l}(b)$ is supported in the set $(-\gamma_N N/2, \gamma_N N/2)\cap \Z$.  Then, Theorem \ref{thm:FiniteQuantBLTHD} applied with $Q=\gamma_N$, gives 
	\[ \frac{C}{R\gamma_N} \ \le\  \frac{1}{N^l} \sum_{|j_k|\ge \frac{NR}{2}} |b(\v{j})|^2,\]
	where the second sum does not appear due to the support condition on $\mathcal{F}_{d,l}(b)$.  As in part (i), let $1 \le S \le \gamma_N$ and $R=\lceil S^{1/\alpha}\rceil$ if $1<p<\infty$ and $R=S$ if $p=1$.  Summing over values of $S$, and applying Lemma \ref{lem:BasicSumBounds} we find 
	\begin{eqnarray*} 
		\frac{C}{2\gamma_N} \sum_{S=1}^{\gamma_N} S^{-\tau} &\le& \frac{1}{N^l} \sum_{S=1}^{\gamma_N} \sum_{|j_k|\ge \frac{NS^{1/p}}{2}} |b(\v{j})|^2
		\ \le\  2^p \alpha^{p,\infty}_k(b),
	\end{eqnarray*}
	and the result follows by combining the constants and another application of equation \eqref{eqn:powersumbounds}. 
\end{proof}

\section{Further Questions}
Upon investigation, similar arguments applied in the one-dimensional Finite BLT apply for several variable analogs. It is interesting to consider the question of whether there are sequences which have the `best' localization properties, those for which the $\alpha_k$ norm is minimized over the set of all $A, B$-Gabor Riesz bases. There is a conjecture \cite{Lammers} of Lammers and Stampe which addresses this question and is still open to the authors' knowledge.
Also of interest is whether uncertainty principles for different continuous basis systems (e.g. \cite{BHW92}) may be discretized to give similar finite dimensional results. 

Another remaining question is related to Theorem \ref{thm:BLTHD}.  In \cite{GHHK}, a more general version of Theorem \ref{thm:SymmetricBLTHD} was shown to hold when $G(g,\Z^{2l})$ is replaced by $G(g,S)$ for any symplectic lattice $S \subset \R^{2l}$.  It is not clear to the authors whether Theorem \ref{thm:BLTHD} also holds in this setting.

\bibliographystyle{amsalpha}

\def\cprime{$'$} \def\cprime{$'$} \def\cprime{$'$} \def\cprime{$'$}
\providecommand{\bysame}{\leavevmode\hbox to3em{\hrulefill}\thinspace}
\providecommand{\MR}{\relax\ifhmode\unskip\space\fi MR }
\providecommand{\MRhref}[2]{%
	\href{http://www.ams.org/mathscinet-getitem?mr=#1}{#2}
}
\providecommand{\href}[2]{#2}

\end{document}